\documentclass[12pt]{article}

\usepackage{cmap}  % should be before 'fontenc'
\usepackage[T2A]{fontenc}
\usepackage[utf8]{inputenc}
\usepackage[english]{babel}
\usepackage[usenames]{color}
\usepackage{amsmath,amssymb,amsthm}
\usepackage[colorlinks=true,linkcolor=blue,urlcolor=red,unicode=true,hyperfootnotes=false,bookmarksnumbered]{hyperref}
\usepackage{mathtools}
\usepackage{xcolor}
\usepackage{csquotes}
\usepackage{wrapfig}
\usepackage{setspace}
\usepackage{enumitem}
\usepackage{appendix}

\newcommand{\eps}{\varepsilon}
\renewcommand{\phi}{\varphi}

\renewcommand{\le}{\leqslant}
\renewcommand{\ge}{\geqslant}
\renewcommand{\mid}{:}

\newcommand{\PP}{\mathbb{P}}

\newcommand{\R}{\mathbb{R}}

\newcommand{\aaa}{\mathcal{A}}
\newcommand{\bb}{\mathcal{B}}
\newcommand{\ff}{\mathcal{F}}

\newcommand{\setCenter}{M}

\newcommand{\residualSet}{R}
\newcommand{\xv}{\mathbf{x}}
\newcommand{\yv}{\mathbf{y}}

\newcommand{\cv}{\mathbf{c}}
\newcommand{\nv}{\mathbf{n}}
\newcommand{\support}{\operatorname{supp}}
\newcommand{\probabilisticCenter}{C^*}
\newcommand{\mSet}{\mathbf{m}}

\newcommand{\solution}{\zeta^*}

\newcommand{\entropy}{\mathbb{H}}
\newcommand{\targetFunction}{s^*(n, \ell, \mSet)}

\newtheorem{theorem}{Theorem}
\newtheorem{problem}{Problem}
\newtheorem{claim}[theorem]{Claim}
\newtheorem{lemma}[theorem]{Lemma}
\newtheorem{corollary}[theorem]{Corollary}
\newtheorem{proposition}[theorem]{Proposition}

\newtheorem{definition}{Definition}

\title{Octopuses in the Boolean cube: families with pairwise small intersections, part I}
\author{
    Andrey Kupavskii\footnote{G-SCOP, CNRS, University Grenoble-Alpes, France and Moscow Institute of Physics and Technology, Russia; Email: {\tt kupavskii@ya.ru}},
    Fedor Noskov\footnote{Moscow Institute of Physics and Technology, Skolkovo Institute of Science and Technology, Higher school of Economics, Russia}
}

\begin{document}

 \maketitle

\begin{abstract}

%The problem we consider originally arises from 2-level polytope theory. This class of polytopes generalizes a number of other polytope families. One of the important questions in this filed can be formulated as follows: is it true for a $d$-dimensional 2-level polytope that the product of the number of its vertices and the number of its $d-1$ dimensional facets is bounded by $d2^{d - 1}$? Recently, Kupavskii and Weltge~\cite{Kupavskii2020} settled this question in positive. Additionally, they prove more general result for families of vectors in $\R^d$ such that the scalar product between any two vectors from different families is either $0$ or $1$.

%Restricted to the Boolean cube, the solution boils down to one elegant step communecated to us by Peter Frankl. Meanwhile, this problem becomes much more sophisticated when we consider several families.

Let $\ff_1, \ldots, \ff_\ell$ be families of subsets of $\{1, \ldots, n\}$. Suppose that for distinct $k, k'$ and arbitrary $F_1 \in \ff_{k}, F_2 \in \ff_{k'}$ we have $|F_1 \cap F_2|\le m.$  What is the maximal value of $|\ff_1|\ldots |\ff_\ell|$?
In this work we find the asymptotic of this product as $n$ tends to infinity for constant $\ell$ and~$m$.

This question is related to a conjecture of Bohn et al. that  arose in the 2-level polytope theory and asked for the largest product of the number of facets and vertices in a two-level polytope. This conjecture was recently resolved by Weltge and the first author.

The main result can be rephrased in terms of colorings. We give an asymptotic answer to the following question. Given an edge coloring of a complete $m$-uniform hypergraph into $\ell$ colors, what is the maximum of $\prod M_i$, where $M_i$ is the number of monochromatic cliques in $i$-th color?
\end{abstract}

%\tableofcontents

%\newpage

\section{Introduction}

A polytope $P \subset \R^d$ is called \textit{2-level} if for each facet $F$ there are two parallel hyperplanes $H,H'$ such that $F\subset H$ and all vertices of $P$ are contained in $H\cup H'$ %either a hyperplane $H$ defined by $F$ or in another hyperplane $H'$ parallel to $F$. It appears that a lot of
Several standard polytope families are $2$-level, e.g. hypercubes, cross-polytopes, simplices. The class of $2$-level polytopes includes a number of important polytopal families like Hanner polytopes, Birkhoff polytopes, the Hansen polytopes and others~\cite{bohn2019enumeration}. These polytopes  arise in such areas of mathematics as the semidefinite programming, communication complexity and polyhedral combinatorics.

A number of authors studied combinatorial structure of 2-level polytopes~\cite{fiorini2016two, aprile20182, bohn2019enumeration, Kupavskii2020}. Bohn et al. \cite{bohn2019enumeration} suggested a beautiful conjecture on the tradeoff between the number of vertices $f_0(P)$ and the number of $d-1$-dimensional facets $f_{d-1}(P)$ of a 2-level polytope $P \subset \R^d$. Concretely, they asked if it is true that $f_0(P) f_{d-1}(P) \le d 2^{d + 1}$ for all $d$. This bound is sharp for cubes and cross-polytopes.  Recently, Kupavskii and Weltge answered this question in the positive~\cite{Kupavskii2020}. Actually, they proved the following variation of the conjecture of Bohn et al., from which it is easy to deduce the original conjecture. For two vectors ${\bf a},{\bf b}\in \R^n$ let $\langle {\bf a},{\bf b}\rangle$ stand for their scalar product.
\begin{theorem}[\cite{Kupavskii2020}]
\label{theorem: kup-weltge}
Let $\aaa$, $\bb$ be families of vectors in $\R^n$ that both linearly span $\R^n$. Suppose that $\langle {\bf a},{\bf b}\rangle\in\{0, 1\}$ holds for all ${\bf a}\in\aaa$, ${\bf b}\in\bb$. Then we have $|\aaa|\cdot|\bb| \le (n + 1) 2^n$.
\end{theorem}
The bound in the theorem is tight since one can take $\aaa = \{\textbf 0, {\bf e}_1, \ldots, {\bf e}_n\}$ and $\bb = \{0, 1\}^n$.

Some of the previous works dealt with the particular case of Theorem~\ref{theorem: kup-weltge} when $\aaa,\bb\subset \{0,1\}^n$. The problem is then much simpler. Actually, we will present a very short and elegant argument due to Peter Frankl that proves Theorem~\ref{theorem: kup-weltge} in the $\{0,1\}$ case.

In this work, we provide the generalization of Theorem~\ref{theorem: kup-weltge} on $\{0,1\}^n$ to several families. Compared to the two families case for $\{0,1\}^n$, this problem becomes much more challenging, and it seems almost hopeless to determine the exact extremal function. The proofs involve some interesting ingredients, such as correlation inequalities for several families. We will say more on these points after we introduce the necessary notation and formulate the main result.
In what follows, we will work with families of sets instead of families of $\{0,1\}$-vectors.

%Restricted to the Boolean cube $\{0, 1\}^n$, Theorem~\ref{theorem: kup-weltge} can be obtained in one step as we will show in the sequel. Meanwhile, it becomes much more sophisticated then we consider product of several families of sets in the Boolean cube, and this problem is the object of our paper. We will provide its precise description as soon as the notation will be introduced.

\subsection{Notation}

%In this work, we deal with families (or collections) of finite sets on the ground set  we consider
Put $[n] = \{1, \ldots, n\}$, and, more generally, $[a, b] = \{a, a + 1, \ldots, b\}$ for positive integers $a,b,n$. Given a set $X$, we denote by $2^{X}$ the set of all subsets of $X$. We denote by ${X\choose k}$ (${X\choose \le k}$) the family of all subsets of $X$ of cardinality $k$ (at most $k$). %Hence, any family of sets is a subset of $2^{[n]}$.
%A family of subsets of $X$ of cardinality $k$ (at most $k$) is denoted by $\binom{X}{k}$ (${X\choose \le k}$). Obviously, $\binom{X}{\le k} = \bigsqcup_{t = 0}^k \binom{X}{t}$, where $\binom{X}{0} = \{\varnothing \}$.
We also denote $\binom{n}{\le m} := \left| \binom{[n]}{\le m} \right| = \sum_{t = 0}^m \binom{n}{t}$.

In this paper, we study families of sets with  the ``$\mSet$-\textit{overlapping property}'', which is defined below.
\begin{definition}
Fix a positive integer $\ell$. Let $\mSet = (\mSet_S)_{S \in \binom{[\ell]}{2}}$ be a vector of non-negative integers indexed by unordered pairs $\{k, k'\} \in {[\ell]\choose 2}$. For simplicity we suppress brackets in $\mSet_{\{k, k'\}}$ and assume that $\mSet_{k, k'}$, $\mSet_{k', k}$, and $\mSet_{\{k, k'\}}$ identify the same entry. Families $\ff_1, \ldots, \ff_{\ell}\subset 2^{[n]}$ {\em satisfy an $\mSet$-overlapping property} if for any distinct $k_1, k_2 \in [\ell]$ and any sets $F_1 \in \ff_{k_1}$, $F_2 \in \ff_{k_2}$ we have
\begin{align*}
    |F_1 \cap F_2| \le \mSet_{k_1, k_2}.
\end{align*}
If $\mSet_{k_1, k_2} = m$  for all pairs $k_1,k_2$  then the property is referred to as {\em $m$-overlapping}, and {\em overlappling} if additionally $m=1$.
\end{definition}
%\todo{TODO: define smol-o notation}

\subsection{Problem statement and results}

In this work, we address the following problem.
\begin{problem}
\label{problem: main}
Let $n, \ell$ be positive integers, $\mSet$ be a vector of $\binom{\ell}{2}$ non-negative integers and $\ff_1, \ldots, \ff_\ell \subset 2^{[n]}$ be families with the $\mSet$-overlapping property. What is the maximal value $s^*(n, \ell, \mSet)$ of the product $|\ff_1|\cdot\ldots\cdot |\ff_\ell|$?
\end{problem}

If all coordinates of $\mSet$ are equal to $m$, we denote $s^*(n, \ell, m) := s^*(n, \ell, \mSet)$.

It is easy to see that $s^*(n, \ell, 0) = 2^n$: indeed, supports of sets in distinct families are disjoint. Recently, Aprile, Cevallos, and Faenza \cite{aprile20182} showed that $s^*(n, 2, 1) = (n + 1) 2^n$. In personal communication, Peter Frankl \cite{Frankl} gave a simple and elegant proof that $s^*(n, 2, t) = 2^n\sum_{i=0}^t \binom{n}{i}$ using Harris--Kleitman correlation inequality (we present his proof in Theorem~\ref{theorem: Frankl's theorem}). In \cite{ryser1974subsets}, Ryser studied a similar question for one family. In particular, he showed that if for $n\notin\{9, 10\}$ $n$ sets of size at least 3 intersect each other in at most 1 element, then they form either a finite projective plane or a symmetric group divisible design.

\medskip

The main result of this paper is the following theorem:

\begin{theorem}
\label{theorem: maximal families structure}
Let $\ell, $ be positive integers and let $\mSet$ be a vector of integers as above.  Then, as $n\to \infty$, we have  the following. %asymptotic of $s^*(n, \ell, \mSet)$ is the following:
\begin{align}
\label{eq: target function asymptotics}
    \targetFunction =
    \left (1 + O \left (\frac{1}{\sqrt{n}} \right ) \right ) \cdot
    2^n\cdot
    \prod_{S\in{[\ell]\choose 2}}
        \bigg(
            \frac 1{\mSet_S!}
            \Big(
                \frac{\mSet_S \cdot n}{
                    \sum_{S'\in {[\ell]\choose 2}}
                        \mSet_{S'}
                }
            \Big)^{\mSet_S}
        \bigg).
\end{align}
\end{theorem}

%As we have briefly mentioned, for $\ell=2$ Peter Frankl found a short and elegant argument that allows to determine $s^*(n,2,t)

Unlike in the case $\ell=2$, it seems extremely challenging to determine the exact behaviour of $\targetFunction$ for general $\ell$. In the follow-up paper \cite{KupaNos2}, we improve the precision of the asymptotic from $O(n^{-1/2})$ to $O(n^{-1})$. More importantly, we will show that all extremal examples must be superfamilies of a certain tuple of families that deliver lower bound in Theorem~\ref{theorem: maximal families structure}. However, even coming up with the right extremal construction for general $\ell$ seems to be very difficult.

From the theorem above we immediately derive a cleaner formula for %While the asymptotic of $\targetFunction$ in general is not  not clear throughout this work,
the asymptotic of $s^*(n, \ell, m)$.
\begin{corollary}
\label{corollary: asympotic of m-uniform target function}
Suppose $\ell$ and $m$ are fixed integers. Then, as $n\to \infty$, we have the following. %the asymptotic of $s^*(n, \ell, m)$ is the following:
\begin{align*}
    s^*(n, \ell, m) = \left (1 + O \left ( \frac{1}{\sqrt{n}}\right ) \right )
%    \left [
%        \sum_{t = 0}^m
%            \binom{n/\binom{\ell}{2}}{t}
%    \right ]^{\binom{\ell}{2}}
%    2^n
%    \sim
\left [
     \frac{1}{m!}
     \left (
        \frac{n}{\binom{\ell}{2}}
     \right )^m
    \right ]^{\binom{\ell}{2}}
    2^n.
\end{align*}
%where $\binom{x}{t}$ is a polynomial of the form:
%\begin{align*}
%    \binom{x}{t} = \frac{
%        x \cdot (x - 1) \cdot \ldots \cdot (x - t + %1)
%    }{t!}.
%\end{align*}
\end{corollary}

\medskip

There is an equivalent formulation of Problem~\ref{problem: main} for $\mSet = (m,m,\ldots, m)$.
\begin{problem}
\label{problem: main_graphs}
Let $n, \ell, m$ be integers and $H$ be a complete $(m + 1)$-uniform hypergraph on $n$ vertices. Take some coloring of  edges of $H$ into $\ell$ colours. Let $k_i$, $i=1, \ldots, \ell$ be the number of monochromatic cliques of colour $i$ in $H$. What is the maximum value $\tilde s(n, \ell, m)$ of $k_1\cdot\ldots\cdot k_\ell$ over all possible colorings? (We assume that each of the sets of size $\le m-1$ forms a monochromatic clique in each color.)
\end{problem}
In particular, $\tilde s(n, 2, 1)$ is the maximum of the product of the number of cliques and the number of independent sets in a graph $G$ on $n$ vertices.
In \cite{aprile20182} it was shown that Problem~\ref{problem: main} and Problem~\ref{problem: main_graphs} for are equivalent for $m = 1$. Generally, the following holds.

\begin{proposition}
\label{proposition: formulation equivalence}
Let $n, \ell$ be integers, then $s^*(n, \ell, m) = \tilde s(n, \ell, m)$.
\end{proposition}

 %For the proof one can see Claim~\ref{claim: hypergraph point of view}.

% \textcolor{red}{Old version}

% \begin{theorem}
% \label{theorem: asymptotic of target function}
% Suppose $\ell$ and $\mSet$ are fixed while $n$ tends to infinity. Then the asymptotic of $\targetFunction$ is the following:
% \begin{align*}
%     \targetFunction = (1 + o(1)) \cdot 2^n\cdot \solution,
% \end{align*}
% where $\solution$ is a solution of the following optimization problem:
% \begin{align*}
%     \solution  = &
%     \max_{\nv}
%     \prod_{S \in \binom{[\ell]}{2}}
%         \left (
%             \sum_{t = 0}^{\mSet_S}
%                 \binom{
%                     \nv_S
%                 }{
%                     t
%                 }
%         \right ), \\
%     \text{ s.t. } &
%     \sum_{S \in \binom{[\ell]}{2}}
%         \nv_S
%     = n, \nonumber \\
%     & \forall S\in \binom{[\ell]}{2} \quad \nv_S \in \Z_+. \nonumber
% \end{align*}
% \end{theorem}

The rest of the paper is organised as follows. In Section~\ref{section: tools} we list some tools that we use and give Peter Frankl's proof that determines $s^*(n,2,m)$. We  prove Proposition~\ref{proposition: formulation equivalence} and discuss related questions in Section~\ref{section: hypergraph point of view}.  In Section~\ref{section: sketch of the proof} we give the sketch of the proof of Theorem~\ref{theorem: maximal families structure}. In Section~\ref{section: lower bound} we prove the lower bound in~\eqref{eq: target function asymptotics}.  In Section~\ref{section: proof of theorem} we prove the upper bound.

In what follows, the standard asymptotic notation such as $f = o(g), f = \Omega(g)$ etc. for some functions $f,g$ is always with respect to $n\to \infty.$

%Throughout this paper, using asympotics notation, i.e. $O(\cdot)$, $\Omega(\cdot)$, etc., we suppose that $\ell$ and $\mSet$ are fixed while $n$ tends to infinity.

\section{Tools}
\label{section: tools}
There is a trivial bijection between $2^{[n]}$ and the Boolean cube $\{0, 1\}^{n}$. Given a set $X$, we consider its characteristic vector $\cv^X$ whose $i$-th coordinate $\cv^X_i$ equals 0 if $i \not \in X$ and $1$ otherwise. We will take these two equivalent points of view on sets, families, etc. interchangeably.

\subsection{Correlation inequalities}

Given a probability measure on the Boolean cube, we can consider a family of subsets as an event. %So, a set will be a sample point, defined by its characteristic vector.
For example, consider a uniform measure on the Boolean cube $\{0, 1\}^n$ and some family of sets $\ff \subset 2^{[n]}$. For a random set $X$ sampled from the uniform measure, the probability of an event $\{X \in \ff\}$ is $\frac{|\ff|}{2^n}$.

This point of view provides us with a range of tools that we use throughout this work. %We will use several different measures for different sets. There are a lot of works on this topic.
In this subsection, we discuss correlation inequalities which are a powerful tool in Combinatorics and Extremal Set Theory. Alon and Spencer in their book~\cite{Alon2016} write that the first appearance of correlation inequalities can probably be attributed to Harris \cite{Harris1960} and Kleitman \cite{Kleitman1966}.

We say that a family of subsets is {\it down-closed (a downset)} if $A\in\ff$ and $B\subset A$ implies $B\in \ff$. Harris--Kleitman correlation inequality is as follows.
\begin{theorem}[Harris--Kleitman correlation inequality]
\label{theorem: Kleitman's inequality}
Let $\aaa, \bb \subset 2^{[n]}$ be down-closed. Then
\begin{align*}
    |\aaa| |\bb| \le 2^n |\aaa \cap \bb|.
\end{align*}
\end{theorem}
If we reformulate the statement in the following way:
\begin{align*}
    \frac{
        |\aaa|
    }{
        2^n
    }
    \cdot
    \frac{
        |\bb|
    }{
        2^n
    }
    \le
    \frac{
        |\aaa \cap \bb|
    }{
        2^n,
    }
\end{align*}
we see that it states that the events $\{X \in \aaa\}$ and $\{X \in \bb\}$ are positively correlated for $X$ that is uniformly distributed over the Boolean cube.

Somewhat surprisingly, this inequality alone can be used to solve our problem in the case of two families. This solution was given by Peter Frankl in a private conversation. We provide it here.
\begin{theorem}[Frankl \cite{Frankl}]
\label{theorem: Frankl's theorem}
Let $\aaa \subset 2^{[n]}$ and $\bb \subset 2^{[n]}$ be such families that for any $A \in \aaa$ and any $B \in \bb$ it holds that $|A \cap B| \le m$. Then
\begin{align*}
    |\aaa||\bb| \le 2^n \sum_{t = 0}^m \binom{n}{t}.
\end{align*}
%Moreover, there are two families $\aaa_0$ and $\bb_0$ such that the equality holds.
\end{theorem}
It is not difficult to see that $\aaa = 2^{[n]}$ and $\bb = {[n]\choose \le m}$ satisfy the conditions of the theorem and attain equality in the inequality above.

%Now we prove Theorem~\ref{theorem: Frankl's theorem}.

\begin{proof}
Consider families $\aaa$ and $\bb$ that maximize $|\aaa||\bb|$. They are down-closed, otherwise consider their down-closures $\aaa^{\downarrow}, \bb^{\downarrow}$, where
\begin{align}
    \ff^{\downarrow} & = \{F \subset [n] \mid \exists F' \in \ff \; F \subset F'\}.
%    \bb^{\downarrow} & = \{F \subset [n] \mid \exists F' \in \bb \; F \subset F'\}.
\end{align}
%Thus, for all sets in a family we add their subsets.
Clearly, $\aaa^{\downarrow}, \bb^{\downarrow}$ satisfy the $m$-overlapping property as well, which by maximality implies $\aaa = \aaa^{\downarrow}, \bb = \bb^{\downarrow}$. %wouldn't be broken. Exactly, if $F'$ is a subset of $F$ and $G'$ is a subset of $G$ then $F' \cap G' \subset F \cap G$ and it leads to contradiction.

Since $\aaa$ and $\bb$ is down-closed, we can apply Theorem~\ref{theorem: Kleitman's inequality}. Then
\begin{align*}
    |\aaa||\bb| \le 2^n |\aaa \cap \bb|.
\end{align*}
But $\aaa \cap \bb$ can consist only of sets of cardinality at most $m$. Thus, $\aaa \cap \bb \subset \binom{[n]}{\le m}$ and
\[
%\pushQED{\qed}
|\aaa||\bb| \le 2^n \sum_{t = 0}^m \binom{n}{t}.\qedhere
\popQED
\]
%The tight example is the following: $\aaa_0 = 2^{[n]}$, $\bb_0 = \binom{[n]}{\le m}$.
\end{proof}

Note that if $m = 1$ then Theorem~\ref{theorem: Frankl's theorem} implies the bound of Aprile et al~\cite{aprile20182} and is a special case of Kupavskii and Weltge's result~\cite{Kupavskii2020}.

Meanwhile, Theorem~\ref{theorem: Kleitman's inequality} is not sufficient to resolve Problem~\ref{problem: main} already for $\ell = 3$. There are several correlation inequalities that generalize Harris--Kleitman correlation inequality. One is  Daykin's inequality~\cite{Daykin1977}. Before we present it, we introduce some extra notation. Given vectors $\xv$ and $\yv$ from $\R^n$  we define vectors $\xv \land \yv$ and $\xv \lor \yv$ such that $(\xv \lor \yv)_j = \max(\xv_j, \yv_j)$ and $(\xv \land \yv)_j = \min(\xv_j,\yv_j)$.

It is easy to see that $\vee$ and $\wedge$ restricted on the Boolean cube relate to union and intersection of sets respectively.
For two families $\ff_1$ and $\ff_2$ we denote by $\ff_1\wedge \ff_2$, $\ff_1\vee \ff_2$ the family of pairwise intersections and pairwise unions of sets from $\ff_1$ and $\ff_2$, respectively: $$\ff_1\wedge \ff_2 = \{F_1\cap F_2: F_1\in \ff_1, \ F_2\in \ff_2\},$$
$$\ff_1\vee \ff_2 = \{F_1\cup F_2: F_1\in \ff_1, \ F_2\in \ff_2\}.$$
Note that for down-closed families $\aaa$ and $\bb$ we have $\aaa \wedge \bb = \aaa \cap \bb$. %where, for a reminder, $\aaa \wedge \bb$ is a pairwise intersection.

Then Daykin's correlation inequality states the following.
\begin{theorem}[Daykin correlation inequality]
\label{theorem: Daykin's inequality}
Let $\aaa$ and $\bb$ be two families of sets. Then
\begin{align*}
    |\aaa| |\bb| \le |\aaa \lor \bb||\aaa \land \bb|.
\end{align*}
\end{theorem}

Actually, Daykin's inequality works in a more general setting which we omit. Fortuin, Kasteleyn and Ginibre proposed  another generalization of the Harris--Kleitman inequality in~\cite{Fortuin1971} to a wide class of log-supermodular measures. We will refer to such measures  as the {\em FKG-measures}. %instead of the uniform one in Theorem~\ref{theorem: Kleitman's inequality}.
\begin{definition}[FKG-measure]
    A $\sigma$-finite (nonegative) measure $\mu$ on $\R^n$ is said to be an FKG measure if $\mu$ has a density function $\phi$ with respect to some product measure $d \sigma$ on $\R^k$, (that is, $d \sigma(\xv) = \prod_{j = 1}^k d \sigma(\xv_j)$, and $d \mu(\xv) = \phi(\xv) d \sigma(\xv)$), where $\phi$ satisfies for all $\xv$ and $\yv$ in $\R^k$
    \begin{align}
        \phi(\xv) \phi(\yv) \le \phi (\xv \land \yv) \phi (\xv \lor \yv).
    \end{align}
\end{definition}
This definition is slightly different from the definition used by Fortuin et al in~\cite{Fortuin1971}. It is taken from~\cite{Rinott1991}, where the authors prove a correlation inequality for several families that we will also need in this work.

% \textcolor{red}{Do we need these general inequalities? maybe only the 0,1-case?}
\begin{theorem}[Rinott---Saks correlation inequality~\cite{Rinott1991}]
\label{theorem: rinott-saks inequality}
    Let $\ell,n$ be positive integers. Let $f_1, f_2, \ldots, f_\ell$ and $g_1, \dots, g_\ell$ be nonnegative real-valued functions defined on $\R^n$ that satisfy following condition: for every sequence $\xv^1, \ldots, \xv^\ell$ of elements from $\R^n$ we have
    \begin{align}
        \prod_{i = 1}^\ell
            f_i(\xv^i)
        \le
        \prod_{i = 1}^\ell
            g_i \left (
                \bigvee_{S \in \binom{[n]}{i}}
                \bigwedge_{j \in S} \xv^j
            \right ).
    \end{align}
    Then, for any FKG-measure $\mu$ on $\R^n$ we have
    \begin{align}
        \prod_{i = 1}^m
            \int_{\R^n}
                f_i(\xv) d \mu(\xv)
        \le
        \prod_{i = 1}^m
            \int_{\R^n}
                g_i(\xv) d \mu(\xv).
    \end{align}
\end{theorem}

With appropriate measures and functions, this inequality implies all of the previous ones we listed. We will need the following $\{0, 1\}$-corollary of Theorem~\ref{theorem: rinott-saks inequality}:
\begin{corollary}[Theorem 4.1 from~\cite{Rinott1991}]
\label{corollary: rinott sets theorem}
For any families of sets $\aaa_1, \ldots, \aaa_\ell$,
\begin{align}
    \prod_{k = 1}^\ell |\aaa_k| \le \prod_{k = 1}^\ell \left | \bigvee_{S \in \binom{[\ell]}{k}} \left (\bigwedge_{s \in S} \aaa_s \right ) \right |.
\end{align}
\end{corollary}

Corollary~\ref{corollary: rinott sets theorem} can be derived by taking integral over the counting (i.e., uniform) measure on $\{0,1\}^n$. The corresponding $f_1,\ldots, f_\ell$ and $g_1,\ldots, g_\ell$ are simply the indicator functions that for a given set indicate if it belongs to the corresponding family: $f_i$ for $\aaa_i$ and $g_j$ for $\bigvee_{S \in \binom{[\ell]}{j}} \left (\bigwedge_{s \in S} \aaa_s \right )$.

According to Theorem~\ref{theorem: rinott-saks inequality}, we can replace the uniform with any FKG-measure. An important example of an FKG-measure is the $p$-biased measure with $0\le p\le 1$ for the sets $X\subset [n]$ and families $\ff\subset 2^{[n]}$, defined as follows.
\begin{align*}
    \mu_p(X) =& p^{|X|}(1 - p)^{n - |X|} = \left (\frac{p }{1 - p} \right )^{|X|} (1 - p)^n,\\
    \mu_p(\ff) =& \sum_{X \in \ff} \mu_p(X).
\end{align*}
We say that two measures $\mu,\mu'$ are {\it proportional} if there is a non-zero constant $C$ such that $\mu(x) = C\mu'(x)$ for any $x$. We get the following corollary which slightly more general than Corollary~\ref{corollary: rinott sets theorem}.
\begin{corollary}
\label{corollary: rinott-saks p-biased}
For any collection of families $\aaa_1, \ldots, \aaa_\ell\subset 2^{[n]}$ and $0\le p\le 1$ we have
\begin{align}
    \prod_{k = 1}^\ell \mu_p(\aaa_k) \le \prod_{k = 1}^\ell \mu_p \left (\bigvee_{S \in \binom{[\ell]}{k}} \left (\bigwedge_{s \in S} \aaa_s \right ) \right ).
\end{align}
Moreover, the same holds for any measure that is proportional to $\mu_p.$
\end{corollary}

\subsection{Entropy}

Another probabilistic tool that we use in our work is entropy. One can find a detailed survey in~\cite{Alon2016}.  Consider a random variable $X$ with finite support. Then the entropy $\entropy [X]$ is defined as
\begin{align}
    \entropy [X] = - \sum_{x \in \support X} \PP [X = x] \log_2 \PP [X = x].
\end{align}
It is a non-negative function that satisfies the following properties:
\begin{claim}
\label{claim: entropy properties}
We have
    \begin{itemize}
        \item[(i)] If $X$ and $Y$ are arbitrary random variables distributed over a finite set, then   $\entropy [X, Y] \le \entropy [X] + \entropy [Y]$.
        \item[(ii)] Consider a random variable $X$ distributed over the sets of a family $\ff$. Then $\entropy [X] \le \log_2 |\ff|$. The equality holds if and only if $X$ is uniformly distributed over $\ff$.
    \end{itemize}
\end{claim}

The proof of both statements can be found in~\cite{Alon2016}.

Another similar function that characterizes the difference between two distributions is the cross-entropy. For the particular application that it has in our work, we define it as follows:
\begin{align*}
    \entropy \left ((p_i)_{i \in S}, (q_i)_{i \in S} \right ) = - \sum_{i \in S} p_i \log_2 q_i,
\end{align*}
where $(p_i)_{i \in S}$, $(q_i)_{i \in S}$ are two arbitrary discrete distributions with the same support $S$.

\begin{proposition}
\label{proposition: cross-entropy}
Given a distribution $(p_i)_{i \in S}$, the cross-entropy as a function of $(q_i)_{i \in S}$ achieves its minimum when distributions $(p_i)_{i \in S}$ and $(q_i)_{i \in S}$ coincide.
\end{proposition}

For the proof see, for example,~\cite{Tao2020}.

\subsection{Coverings and matchings}

Our problem has a natural hypergraph interpretation, and so we will need some simple tools from hypergraph theory. %Before we formulate necessary proposition, we introduce two definitions.
We call a subset $T$ of $V$ a \textit{covering} for a hypergraph $(V, E)$, if for any edge $e \in E$ we have $e \cap T \neq \varnothing$. A subset $\mathcal{M}$ of $E$ is a \textit{matching} if edges of $\mathcal{M}$ are pairwise disjoint. A matching $\mathcal{M}$ is \textit{maximal} if it is impossible to enlarge it by adding another $e \in E$. A covering $T$ is called \textit{minimum} if there is no covering of smaller cardinality.

The following statement is folklore.
\begin{proposition}
\label{proposition: covering and matching relation}
Let $(V, E)$ be an arbitrary hypergraph with edges of size at most $t$. Then the size of a minimum covering $T$ is at most size of any maximal matching $\mathcal{M}$ times $t$.
\end{proposition}

\begin{proof}
Note that $\sqcup \mathcal{M}$ is a covering. Otherwise, if an edge is not covered by $\sqcup \mathcal{M}$ then we can add it to $\mathcal{M}$, contradicting the maximality of $\mathcal M$. The cardinality of $\sqcup \mathcal{M}$ is $t |\mathcal{M}|$. Since $T$ is a minimum covering, $|T| \le t |\mathcal{M}|$.
\end{proof}

\section{Counting monochromatic sets}
\label{section: hypergraph point of view}

\begin{wrapfigure}[14]{l}{0.45\textwidth}
    \label{fig: graph coloring}
    \includegraphics[width=0.4\textwidth]{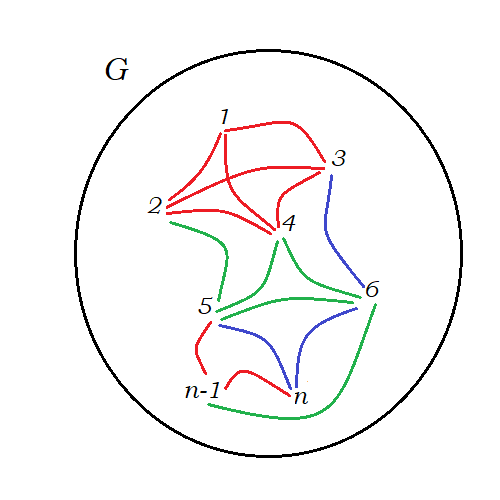}
    \caption{Graph coloring}
\end{wrapfigure}

First, consider the case  $\mSet_S =  1$ for any $S \in \binom{\ell}{2}$. We claim that there is a correspondence between our families and cliques in a colored graph. This connection %is quite trivial and
was previously discussed in~\cite{aprile20182}. To illustrate it, we need the following claim:

\begin{claim}
\label{claim: graph point of view}
Let $\ff_1, \ldots, \ff_l$ satisfy the overlapping property and that are, moreover, maximal w.r.t. this property. Then for each $\ff_k$ it holds that
\begin{itemize}
    \item a subset $K$ of $[n]$ is contained in $\ff_k$ if and only if for each pair $i, j \in K$ $\{i, j\}$ is contained in $\ff_k$,
    \item empty set and every  singleton from $[n]$ belongs to $\ff_k$,
\end{itemize}
and, moreover, for each $\{i, j\} \subset [n]$ there is a unique $k'$ such that $\{i, j\} \in \ff_{k'}$.
\end{claim}

\begin{proof}
First, if $K$ belongs to $\ff_k$ then each pair $\{i, j\}$ from $K$ belongs to $\ff_k$ because of down-closeness. Conversely, suppose that every $\{i, j\}\subset K$ is contained in $\ff_k$. Then for each $k'\ne k$ and a set $F\in \ff_{k'}$ we have $|F\cap K|\le 1$. Therefore, by maximality of $\ff_1,\ldots, \ff_\ell$, $K\in \ff_k$. %set adding $K$ breaks overlapping property and, therefore, there is $k' \neq k$ and set $K' \in \ff_{k'}$ such that $|K' \cap K| \ge 2$. Thus, there is a pair $\{i_0, j_0\} \subset K'$ that belongs to $\ff_k$ and it leads to contradiction.

Second, adding the empty set and singletons does not break the overlapping property, so by maximality they must belong to each family.

Third, if $\{i, j\}$ does not belong to any $\ff_k$ then adding it to any of $\ff_i$ does not break the overlapping property, and so by maximality each $\{i,j\}$ must belong to some $\ff_k$.
\end{proof}

Next, let us construct a coloring of the complete graph $K_n$ on the vertex set $[n]$ into $\ell$ colors based on a maximal collection of families $\ff_1,\ldots, \ff_\ell$ with overlapping property. Color the edge $\{i, j\}$ with color $k$ iff $\{i, j\} \in \ff_k$ and put $E_k = \left \{\{i, j\} : i \neq j \text{ and } \{i, j\} \in \ff_k \right \}$. %This way, we associate to each maximal example a coloring of $K_n$, moreover,
Note that the cardinality of $\ff_k$ is equal to the number of cliques in the graph $G_k := ([n], E_k)$, induced by the edges of $k$-th color. Conversely, any coloring of any graph $G$ on vertex set $[n]$ is associated with some families of subsets of $[n]$ with overlapping property. For example, the coloring on Figure~\ref{fig: graph coloring} produces the following families of sets:
\begin{align*}
    {\color{red} \ff_1} & = \binom{[n]}{\le 1}  \cup 2^{\{1, 2, 3, 4\}} \cup \{\{5, n-1\}, \{n-1, n\}\}, \\
    {\color{green} \ff_2} & = \binom{[n]}{\le 1}  \cup 2^{\{4, 5, 6\}} \cup \{\{6, n-1\}, \{2, 5\}\}, \\
    {\color{blue} \ff_3} & = \binom{[n]}{\le 1}  \cup \{\{5, n\}, \{6, n\}, \{3, 6\}\}.
\end{align*}

Let us briefly discuss how to generalize this correspondence to the $\mSet$-overlapping setting.

Previously, we could as well consider the complement of $G_k$ and  and count independent sets in this complement. It is natural to generalize this point of view. Consider a maximal collection $\ff_1,\ldots,\ff_\ell$ that are $\mSet$-overlapping. For each family $\ff_k$ we construct the following hypergraph:
\begin{align*}
    H_k = \left (
        [n],
        \bigcup_{k' \in [\ell] \setminus \{k\}}
            \ff_{k'}^{(\mSet_{k, k'} + 1)}
    \right ),
\end{align*}
where $\ff^{(t)}$ denotes $\ff \cap \binom{[n]}{t}$.

We call  $I\subset [n]$ \textit{independent} in the hypergraph $H_k$ if it does not contain any edge of $H_k$. %We call $\ff_k$ \textit{maximal} if we cannot add any subset of $[n]$ to $\ff_k$ without breaking the $\mSet$-overlapping property.
We claim that the number of independent sets in $H_k$ equals $|\ff_k|$.

\begin{claim}
\label{claim: hypergraph point of view}
    If $\ff_1,\ldots,\ff_\ell$ are maximal and $\mSet$-overlapping then $\ff_k$ consists of all independent sets in $H_k$, where $H_k$ is defined above.
\end{claim}

\begin{proof}[Sketch of the proof]
It follows from two implications:
\begin{enumerate}
    \item If $I$ is independent in $H_k$ then it intersects any set of $\ff_{k'}$, $k' \in [\ell]\setminus \{k\}$, in at most $\mSet_{k, k'}$ elements. Thus, it is contained in $\ff_k$ due to maximality.
    \item If $F \in \ff_k$ then it intersects any set of $\ff_{k'}$, $k' \in [\ell]\setminus \{k\}$, in at most $\mSet_{k, k'}$ elements. Thus, $F$ is an independent set in the hypergraph $H_k$.
\end{enumerate}
\end{proof}

%Obviously, families of subsets in the extremal example are maximal.

In addition, if all entries of $\mSet$ are equal to $m$, then $H_k$ becomes the complement to the subhypergraph of $K^{m + 1}_n$ consisting of all edges colored into color $k$. This proves Proposition~\ref{proposition: formulation equivalence}. (Recall that we assume that any $m$ or less vertices form a monochromatic clique in any color.)

\section{Sketch of the proof of  Theorem~\ref{theorem: maximal families structure}}
\label{section: sketch of the proof}

\begin{wrapfigure}[16]{r}{0.45\textwidth}
    \centering
    \includegraphics[width=0.4\textwidth]{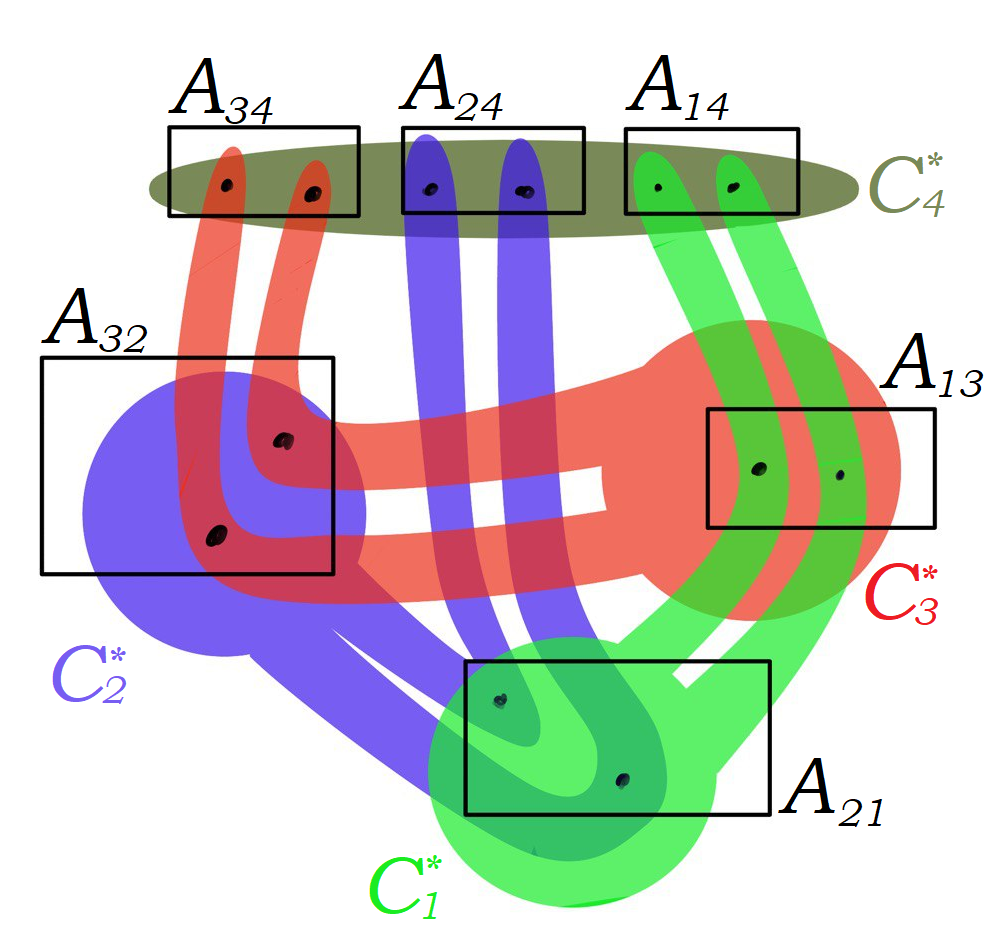}
    \caption{Octopuses described in Section~\ref{section: sketch of the proof}.}
    \label{fig: octopuses}
\end{wrapfigure}

Throughout this section, we assume that $\ff_1,\ldots, \ff_\ell$ are $\mSet$-overlapping and extremal (that is, maximize the product of cardinalities).

 First, we find the lower bound that matches the asymptotic in~\eqref{eq: target function asymptotics}.
The construction of the example is based on the following guess: since $\targetFunction$ is proportional to $2^n$, each $\ff_k$, $k \in [\ell]$ should contain some ``center'' set $\probabilisticCenter_k$ such that $\ff_k$ contains all subsets of $\probabilisticCenter$ and these sets together cover $[n]$  completely. That guarantees that $2^{\probabilisticCenter_k} \subset \ff_k$ and that the ``exponential part'' of the product $\prod_{k \in [\ell]} |\ff_k|$ is $2^n$. The polynomial part of $|\ff_k|$ arises from concatenations of subsets from $\probabilisticCenter_k$ and elements from other centers. That makes a family $\ff_k$ look like an ``octopus'' with ``body'' $\probabilisticCenter_k$ and ``tentacles'' directed to the centers of others. For instance, in the case of $m = 1$, a family $\ff_1$ presented on Figure~\ref{fig: octopuses} can be decomposed as follows:
\begin{align*}
    {\color{green} \ff_1} \simeq 2^{\color{green} \probabilisticCenter_1} \vee \binom{A_{13}}{\le 1} \vee \binom{A_{14}}{\le 1}
\end{align*}
with the body $\probabilisticCenter_1$ and tentacles $\binom{A_{13}}{ \le 1} \vee \binom{A_{14}}{\le 1}$. We use ``$\simeq$'' instead of ``$=$'' because in extremal examples we need to add all small sets, but they typically account for  a negligibly small fraction of the family.

This summarizes the rough structure on which we based the example for the lower bound in Section~\ref{section: lower bound}. In what follows, we discuss the upper bound.

The proof of the upper bound is based on a bootstrapping idea: establishing the asymptotic helps to obtain understanding of the structure of extremal examples and vice versa. First, we prove that a set $M_k$ of maximal cardinality from each family $\ff_k$ can be considered as a proxy of the center $\probabilisticCenter_k$. More precisely, we prove that $\left |[n]\setminus \bigcup_{k \in [\ell]} M_k \right | = O(\log n)$ using Rinott---Saks inequality for $p$-biased measures (Corollary~\ref{corollary: rinott-saks p-biased}). The detailed argument is given in Lemma~\ref{lemma: log residual}.

In what follows, we will use the definition of a (normalized) \textit{degree} of a set with respect to some family. The degree $d_k$ of a set $F$ in a family $\ff_k$ is
\begin{align*}
    d_k(F) := \frac{|\{F' \in \ff_k \mid F \subset F'\}|}{|\ff_k|}.
\end{align*}
The degree $d_k(x)$ of an element $x \in [n]$ is just $d_k(\{x\})$.

We next show that an inductive application of Daykin's inequality (Theorem~\ref{theorem: Daykin's inequality}) over $k\in \ell$ delivers the asympotic of $\targetFunction$ up to a constant factor. It allows to have a good control on the  degrees: given a subset $K \subset [\ell]$ and sets $F_k \in \ff_k$, $k \in K$, we have
\begin{align}
\label{eq: general degree tradeoff}
    \prod_{k \in K} d_k(F_k) = O \left ( n^{- \sum_{\{k, k'\} \in \binom{K}{2}} |F_k \cap F_{k'}|}\right ),
\end{align}
and, in particular,
\begin{align}
    \label{eq: degree tradeoff}
    d_{k_1}(x) d_{k_2}(x) d_{k_3}(x) = O(n^{-3})
\end{align}
for any distinct $k_1,k_2, k_3$ and $x \in [n]$.
Indeed, it is easy to see that families $\ff_k(F_k)$, $\ff_{k'}$ for $k \in K$ and $k' \in [n] \setminus K$ are $\mSet'$-overlapping for a suitable $\mSet'\le \mSet$. Using the upper bound for  $s^*(n, \ell, \mSet')$ that follows from iterative Daykin's inequality applications and the lower bound for $\targetFunction$ allows to obtain~\eqref{eq: general degree tradeoff} and~\eqref{eq: degree tradeoff}.

The entropy argument of Proposition~\ref{proposition: big x fraction} guarantees that most of the elements of $[n]$ have positive constant degrees in some $\ff_k$. If we denote the $k$-th least normalized degree by $d_{(k)}(x)$ and the index of corresponding family by $(k)(x)$, we notice that for most of elements in $[n]$, $d_{(l)}(x)$ has constant lower bound, and, consequently, $d_{(l - 2)}(x) = O(n^{-3/2})$ due to~\eqref{eq: degree tradeoff}. In this way, removing suitable sets, we are able to prune families $\ff_k$ such that their size changes by a factor $\left (1 - O(n^{-1/2}) \right )$  and $d_k(x) = 0$ if $k \not \in \{(\ell)(x), (\ell - 1)(x)\}$. In other words, in the modified families each element has non-zero degree only in two families that correspond to two initial families in which $d_k(x)$ was   the first and the second largest. While the set of maximal cardinality $M_k$ is a proxy of the octopus's body of $\ff_k$, the set $\{x \in [n] \mid k = (l-1)(x)\}$ is a proxy of its tentacles.

We denote these pruned families by $\ff_k'$. By construction, for any $K \subset [\ell]$ of cardinality greater than 2, we have $\bigwedge_{k \in K} \ff_k' = \{\varnothing\}$. That significantly simplifies the Rinott-Saks inequality for families of sets (Corollary~\ref{corollary: rinott sets theorem}), since among multipliers from the right-hand side only two factors remain. They can be bounded as follows:
\begin{align}
\label{eq: upper bound in proof sketch}
    \left | \bigvee_{S \in \binom{[\ell]}{2}} \bigwedge_{s \in S} \ff_s' \right | & \le \prod_{S \in \binom{[\ell]}{2}} \left | \binom{\support \wedge_{s \in S} \ff_{s}'}{\le \mSet_S} \right |, \\
    \left | \bigvee_{k \in [\ell]} \ff_k' \right | & \le 2^n. \nonumber
\end{align}
Optimizing over sizes of disjoint sets $\support(\wedge_{s \in S} \ff_s')$, we obtain tight upper bound of $\targetFunction$ up to a factor $1 + O(n^{-1/2})$. (Note that the error term $O(n^{-1/2})$ is an artefact of the pruning that we did.)

In the follow-up paper~\cite{KupaNos2}, we will improve the error term and determine a bulk of the structure of extremal examples. The tentacle analogy is very useful in understanding their structure.

\section{Proof of the lower bound}
\label{section: lower bound}

In this section, we provide a construction that gives the lower bound in Theorem~\ref{theorem: maximal families structure}.

\begin{theorem}
\label{theorem: general lower bound}
There are families $\ff_1, \ldots, \ff_\ell \in 2^{[n]}$ satisfying $\mSet$-overlapping property such that
\begin{align*}
    \prod_{k = 1}^\ell
        |\ff_k|
    =
    \left (1+O(n^{-1}) \right ) 2^n \cdot
    \prod_{S\in{[\ell]\choose 2}}
        \bigg(
            \frac 1{\mSet_S!}
            \Big(
                \frac{
                    \mSet_S\cdot n
                }{\sigma}
            \Big)^{\mSet_S}
        \bigg) = \left (1+O(n^{-1}) \right ) C n^{\sigma}2^{n},
\end{align*}
where $\sigma = \sum_{S\in {[\ell]\choose 2}} \mSet_S$ and $C$ is a constant depending on $\ell$ and $\mSet$ only.
\end{theorem}

\begin{table}
    \centering
    \begin{tabular}{c c | l}
         $\probabilisticCenter_{\ell - 1}$ & $=$ & $A_{\ell-1, \ell}$  \\
         $\probabilisticCenter_{\ell - 2}$ & $=$ & $A_{\ell-2, \ell} \cup A_{\ell-2, \ell-1}$  \\
         $\ldots $ & & \quad \ldots  \\
         $\probabilisticCenter_2$ & $=$ & \, $A_{2, \ell} \,\, \cup \, A_{2, \ell-1} \, \cup \ldots \cup A_{2, 3}$  \\
         $\probabilisticCenter_1$ & $=$ & \, $A_{1, \ell} \,\, \cup  \, A_{1, \ell-1} \, \cup \ldots \cup A_{1, 3} \cup A_{1, 2}$ \\
         \hline
         & & $ \;\; \ff_\ell \quad \quad \, \ff_{\ell - 1} \quad \;\;\, \ldots \quad \ff_{3} \quad \quad \ff_2$
    \end{tabular}
    \caption{In our example for Theorem~\ref{theorem: general lower bound} each set in the family $\ff_k$ consists of two parts: an arbitrary subset of the  center $\probabilisticCenter_k = \bigcup_{k' > k} A_{k, k'}$ (note that $\probabilisticCenter_\ell = \varnothing$) and subsets (``tentacles'') of size at most $\mSet_{k',k}$ in each of the domains $A_{k',k},$ $k'<k$. %, where each tenta $\bigvee_{k' < k} \binom{A_{k', k}}{\mSet_{k', k}}$ which the family $\ff_k$ puts into else domains.
    The rows of the table above are indexed by the corresponding domains of the families and the columns are indexed by the families, where the column indexed by $\ff_i$ consists of the ``target sets'' of the tentacles of sets from $\ff_i$. Note that $\ff_1$ has no tentacles.}
    \label{tab:my_label}
\end{table}

\begin{proof}
Consider some vector $\nv$ with coordinates indexed by $S$,  $S\in {[\ell]\choose 2}$, and a partition  of the set $[n]$ into $\binom{l}{2}$ sets $A_{k, k'}$, $1 \le k < k' \le \ell$, such that $|A_{k, k'}| = \nv_{k, k'}$. Then,
define (cf. Table~\ref{tab:my_label})
\begin{align*}
    \ff_k =
    \left (
        2^{\bigcup_{k' > k} A_{k, k'}}
    \right )
    \lor
    \bigvee_{k' < k} \binom{
        A_{k', k}
    }{
        \le \mSet_{k, k'}
    }.
\end{align*}
Obviously, if $F \in \ff_{k_1}$, $G \in \ff_{k_2}$ and $k_1 < k_2$, then the intersection of $F$ and $G$ is contained  in $A_{k_1, k_2}$. At the same time, by definition, each set from  $\ff_{k_2}$ contains at most $\mSet_{k_1,k_2}$ elements in $A_{k_1,k_2}$. %has no intersection outside of $A_{k_1, k_2}$. Meanwhile, $\ff_{k_2}$ contains only such sets from $A_{k_1, k_2}$ that have cardinality at most $\mSet_{k_1, k_2}$.
Thus, $\ff_1,\ldots,\ff_\ell$ satisfy the $\mSet$-overlapping property.

It is easy to see that
\begin{align*}
    |\ff_k| = \prod_{k' < k} \binom{|A_{k', k}|}{\le \mSet_{k', k}}
    2^{\sum_{k' > k} |A_{k, k'}|},
\end{align*}
and, consequently,
\begin{align*}
    \prod_{k = 1}^\ell
        |\ff_k|
    =
    \prod_{S \in \binom{[\ell]}{2}}
        \binom{\nv_S}{\le \mSet_S}
    2^n.
\end{align*}
Maximizing over $\nv$ delivers the following optimization problem:
\begin{align}
    & \max_{\nv}
    \prod_{S \in \binom{[\ell]}{2}}
        \left (
            \sum_{t = 0}^{\mSet_S}
                \binom{
                    \nv_S
                }{
                    t
                }
        \right ), \label{eq: optimization problem} \\
    & \text{ s.t. }
    \sum_{S \in \binom{[\ell]}{2}}
        \nv_S
    = n, \nonumber
\end{align}
where $\nv$ is a vector of non-negative integers. To determine the asymptotic of the solution, note that
\begin{align*}
    \sum_{t = 0}^{\mSet_S}
        {\nv_S \choose t}
    \sim
    {\nv_S  \choose \mSet_S}
    \sim
    \frac{
        \nv_S^{\mSet_S}
    }{
        \mSet_S !
    }.
\end{align*}
Maximizing the product of these expressions is equivalent to maximizing the sum of their logarithms. Thus, ignoring lower order terms, the target function of the optimization problem~\eqref{eq: optimization problem} can be changed to
\begin{align*}
    \max_{\nv}
        \sum_{S \in {[\ell] \choose 2}}
            \frac{\mSet_S}{\sigma} \log \frac{\nv_S}{n}.
\end{align*}
The last expression is the minus cross-entropy between discrete distributions $\left (\mSet_S/\sigma \right )_{S \in {[\ell] \choose 2}}$ and $(\nv_S / n)_{S \in {[\ell] \choose 2}}$. By Proposition~\ref{proposition: cross-entropy}, its maximum is achieved when distributions coincide, which proves the lower bound and, moreover, shows that the corresponding example is optimal in the class of examples that we considered.
%and, thus, the asymptotic of the lower bound is obtained.
\end{proof}

\section{Proof of the upper bound}
\label{section: proof of theorem}

%The scheme of proof is as  follows. We consider the extremal example and prove that for each family $\ff_k$ there is a set $M_k$ such that  the union $\bigcup_{k = 1}^\ell M_k$ covers $[n]$ almost completely. Then, we prove that normalized degree $d_k(x)$ of an element $x$ of $M_k$:
%\begin{align*}
%    d_k(x) = \frac{
%        |\{F \in \ff_k \mid x \in F\}|
%    }{
%        |\ff_k|
%    }
%\end{align*}
%has a large value. It implies that if $x \not \in M_k$, then $d_k(x)$ is small. Finally, we prune our families so that each $x$ from $[n]$ belongs to only two families. Then, we can apply Renott-Sacks inequality to the pruned families to obtain the asymptotically tight estimation of $\targetFunction$.

% \todo{Write a sketch of the proof}

We employ the following standard notation for a family $\ff$ and sets $A\subset B:$
$$\ff|_B:=\{F\cap B: F\in \ff\},$$
$$\ff(A,B):=\{F \setminus B : F \in \ff \text{ and } F\cap B = A\},$$
$$\ff(A) := \ff(A, A),$$
$$\ff \left (\overline{A} \right ):= \ff(\varnothing, A).$$
When dealing with singletons, we suppress brackets for simplicity, i.e. $\ff(x) = \ff(\{x\})$ and $\ff(\overline{x})=\ff \left (\overline{\{x\}} \right )$.

In what follows, we  work with families $\ff_1,\ldots,\ff_\ell$ that are $\mSet$-overlapping and that are extremal, i.e., that maximize the product. Due to extremality, they possess certain useful properties, in particular, they must be down-closed. We call a collection of families $\ff_{k}, k \in K \subset [\ell]$ \textit{extremal} if they arise in some extremal example.

\subsection{Maximal sets cover $[n]$ almost completely}

\begin{lemma}
\label{lemma: log residual}
Let $\ff_1, \ldots, \ff_\ell$ be a collection of $\mSet$-overlapping extremal families and let $\setCenter_k \in \ff_k$, $k \in [\ell]$ be the sets of maximal cardinality in the respective families. Define $\residualSet = [n] \setminus \bigcup_{k = 1}^l \setCenter_k$. Then there is a constant $C$ such that
\begin{align*}
    |\residualSet| \le C \log_2 n.
\end{align*}
\end{lemma}

\begin{proof}
We have the following decomposition for $\ff_k$:
\begin{align*}
    |\ff_k| = \sum_{F \in \ff_k|_{\residualSet}} |\ff_k(F, \residualSet)|.
\end{align*}

%Given $F \in \ff_k|_\residualSet$, consider a family $\ff_{k, F}$ defined as follows: %of subsets of $C_k$ such that for each $G \in \ff$: $G \cup F \in \ff_k$. More precisely,
%\begin{align*}
%    \ff_{k, F} = .
%\end{align*}
It follows from the definitions that
\begin{align*}
    |\ff_k(F, \residualSet)| \le \big| \ff_k(F, \residualSet)|_{\setCenter_k} \big|
    \left |
        \ff_k|_{\bigcup_i \setCenter_i \setminus \setCenter_k}
    \right |.
\end{align*}
To bound the size of $\ff_k(F, \residualSet)|_{\setCenter_k}$, consider a hypergraph $H$:
\begin{align*}
    H := \left (
        F \cup \setCenter_k,
        \bigcup_{k'\in [\ell]\setminus \{k\}}
            \left (
                \ff_{k'}|_{F \cup \setCenter_k}
            \right )^{(\mSet_{k, k'} + 1)}
    \right )
\end{align*}
and its induced hypergraph $H'$:
\begin{align*}
    H' := \left (
        \setCenter_k,
        E(H)|_{\setCenter_k}
    \right ).
\end{align*}
Any edge of $H$ intersects both $F$ and $\setCenter_k$ since both $F$ and $\setCenter_k$ are contained in $\ff_k$, and thus cannot contain a set of size $m_{k,k'}+1$ from $\ff_{k'}$. %, otherwise, it lies completely inside $F$ or $\setCenter_k$, and that is a contradiction due to Claim~\ref{claim: hypergraph point of view}.
Consider a vertex cover $T$ for $H'$ and consider a set $(\setCenter_k \setminus T) \cup F$. This set does not contain any edge from $H$ and hence it should belong to $\ff_k$ due to maximality. At the same time, its size is equal to $|\setCenter_k| + |F| - |T|,$ which is at most $|\setCenter_k|$ due to maximality of $\setCenter_k$. Consequently, $|T| \ge |F|$, i.e. covering number of $H'$ is at least $|F|$.

Slightly abusing notation, put $m = \max_{k, k'} \mSet_{k, k'}$. Any edge of $H'$ has size at most $m$ and the vertices of any maximal matching in $H'$ form a vertex cover for $H$. Thus the size of the largest matching in $H'$ is at least $|F|/m$ by Proposition~\ref{proposition: covering and matching relation}. Denote one such matching by $\mathcal{M} \subset E(H')$. We  bound from above $|\ff_k(F, \residualSet)|_{\setCenter_k}|$ using that none of the sets in $\ff_k(F, \residualSet)|_{\setCenter_k}$ can contain an edge from $\mathcal M$.
\begin{align*}
    |\ff_k(F, \residualSet)|_{\setCenter_k}|
    \le
    2^{|\setCenter_k| - |\bigsqcup \mathcal{M}|}
    \prod_{e \in \mathcal{M}}
    (2^{|e|} - 1)
    =
    2^{|\setCenter_k|}
    \prod_{e \in \mathcal{M}}
    (1 - 2^{-|e|})
    \le
    (1 - 2^{-m})^{\frac{|F|}{m}} 2^{|\setCenter_k|}.
\end{align*}
%since $\ff_{k, F}$ doesn't contain any edge from $\mathcal{M}$ completely and all edges of $\mathcal{M}$ are disjoint.

At the same time, it is easy to see that $F\in \ff_k$ satisfies $|F\cap \setCenter_i|\le \mSet_{i,k}$ for every $i\ne k$, and thus $\ff_{k}|_{\bigcup_i \setCenter_i \setminus \setCenter_k}$ has cardinality at most $n^{m(\ell - 1)}$. Thus, the size of $\ff_k$ can be bounded as follows:
\begin{align*}
    |\ff_k|
    \le
    n^{m(\ell-1)}
    2^{|\setCenter_k|}
    \sum_{F \in \ff|_{\residualSet}}
        (1 - 2^{-m})^{
            \frac{|F|}{m}
        }.
\end{align*}
Denote $(1 - 2^{-m})^{\frac{1}{m}}$ by $\varepsilon_m$ and note that $\varepsilon_m<1$ is some constant depending on $m$ only.

Consequently, we can bound the product as follows
\begin{align*}
    \prod_{k = 1}^\ell
        |\ff_k|
    \le
    n^{m \ell (\ell -1)}
    2^{\sum_{k = 1}^\ell |\setCenter_k|}
    \prod_{k = 1}^\ell
    \left (
        \sum_{F \in \ff_{k}|_{\residualSet}}
            \varepsilon_m^{|F|}
    \right ).
\end{align*}
It is easy to see that $\sum_{k = 1}^\ell |\setCenter_k| \le |\bigcup_{k = 1}^\ell \setCenter_k| + \frac{m \ell (\ell - 1)}{2}$, and, thus,
\begin{align}
    \prod_{k = 1}^\ell
        |\ff_k|
    & \le
    n^{m \ell (\ell -1)}
    2^{|\bigcup_{k = 1}^\ell \setCenter_k| + \frac{m \ell (\ell-1)}{2}}
    \prod_{k = 1}^\ell
    \left (
        \sum_{F \in \ff_{k}|_{\residualSet}}
            \varepsilon_m^{|F|}
    \right ) \notag \\
    &
   \label{eq11} \le
    n^{m \ell (\ell -1)}
    2^{n - |\residualSet| + \frac{m \ell(\ell-1)}{2}}
    \prod_{k = 1}^\ell
    \left (
        \sum_{F \in \ff_{k}|_{\residualSet}}
            \varepsilon_m^{|F|}
    \right ).
\end{align}

\begin{figure}
    \label{fig: bipartite hypergraph}
    \centering
    \includegraphics[width=0.5\textwidth]{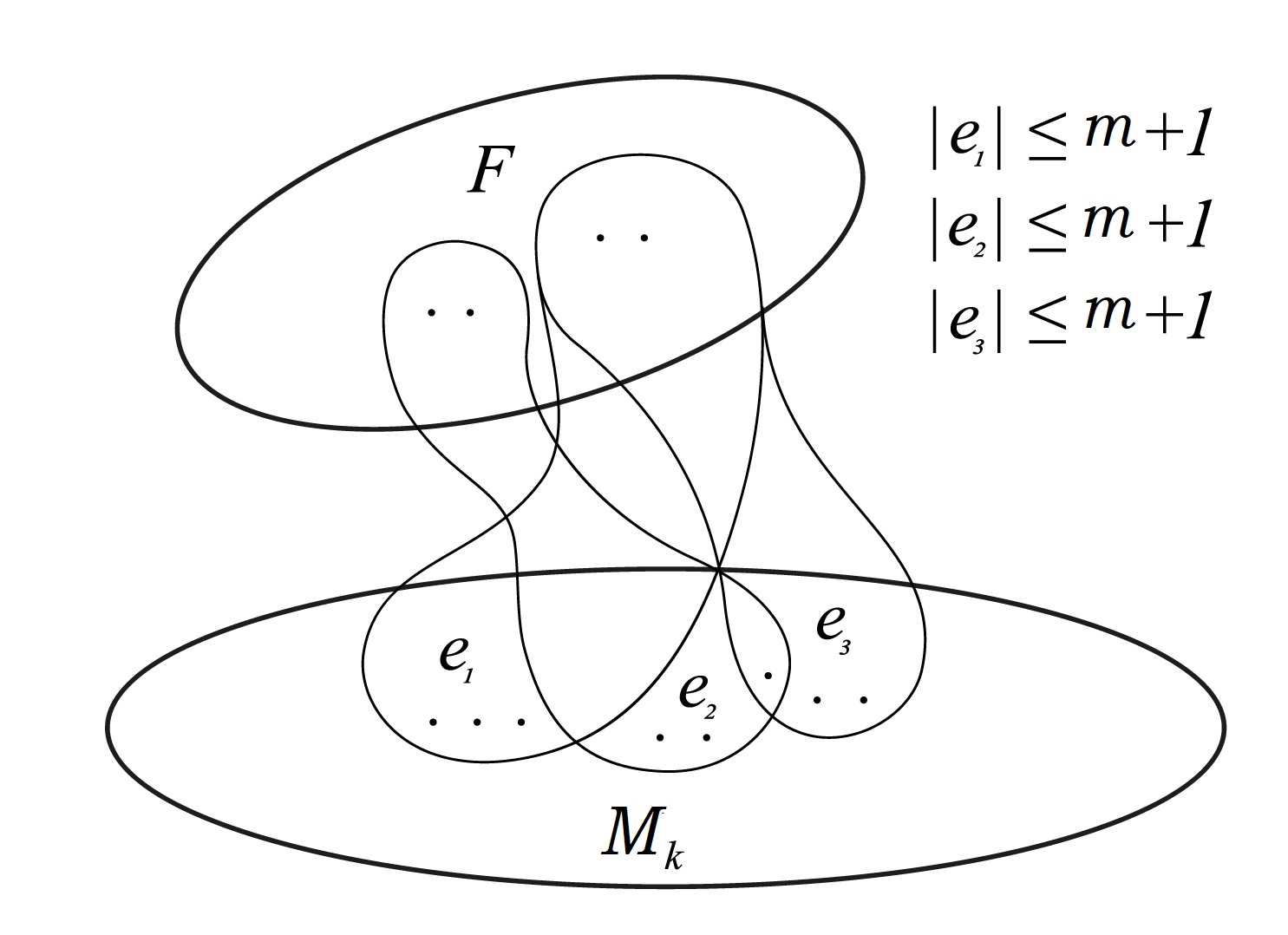}
    \caption{A hypergraph from the proof of Lemma~\ref{lemma: log residual}.}
\end{figure}

We neeed to bound the last product. To this end, note that the function that assigns to a set $F$ the value %$\xv \in \{0, 1\}^{|\residualSet|}$ a value
$(\varepsilon_m)^{|F|}$
is proportional to the $p$-biased measure $\mu_p$ with some $p$, and thus we can apply Corollary~\ref{corollary: rinott-saks p-biased} to it. Put
\begin{align*}
    \mathcal R^{[k]} := \bigvee_{K \in {[\ell] \choose k}} \bigwedge_{k' \in K} \big(\ff_{k'}\big)\big|_{R}.
\end{align*}
Note that
\begin{align*}
    \bigwedge_{k' \in K} \big(\ff_{k'}\big)\big|_{R} = \bigcap_{k' \in K} \big(\ff_{k'}\big)\big|_{R} \subset \binom{R}{\le m},
\end{align*}
and, consequently,
\begin{align*}
    \mathcal R^{[k]} \subset \binom{R}{\le {\ell \choose k} m}.
\end{align*}
Thus, using Corollary~\ref{corollary: rinott-saks p-biased} in the first inequality below, we get
\begin{align*}
    \prod_{k = 1}^\ell
        \sum_{F \in \ff_k|_\residualSet}
            \varepsilon_m^{|F|}
    & \le
    \prod_{k = 1}^\ell
    \sum_{F \in \mathcal R^{[k]}}
        \varepsilon_m^{|F|}
    \le
    \left (
        \sum_{i = 0}^{|\residualSet|}
            \binom{|\residualSet|}{i}
            \varepsilon_m^{i}
    \right )
    \prod_{k = 2}^\ell
        \binom{|R|}{\le {\ell \choose k} m} \\
    & =
    \left (
        1 +
        \varepsilon_m
    \right )^{|\residualSet|}
    \cdot
    \prod_{k = 2}^\ell
        \binom{|R|}{\le {\ell \choose k} m}.
\end{align*}
Due to Theorem~\ref{theorem: general lower bound}, $\prod_{k = 1}^\ell |\ff_k| = \Theta( n^{\sigma} 2^n)$, where $\sigma = \sum_{S \in \binom{[\ell]}{2}} \mSet_{S}$. Thus, combining the above with \eqref{eq11}, we get
\begin{align*}
    \Theta(n^\sigma 2^n)
    & =
    |R|^{O(1)}
    \left (
        1 +
        \varepsilon_m
    \right )^{|\residualSet|}
    2^{n - |\residualSet|}, \\
    \left (
        \frac 2 {1 + \varepsilon_m}
    \right )^{|\residualSet|}
    & = n^{O(1)}.
\end{align*}
Since $\varepsilon_m < 1$, the last inequality implies that
\begin{align*}
    |\residualSet|
    =
    O \left (
        \log n
    \right ).
\end{align*}
\end{proof}

\subsection{Weak upper bound}

In this section, we give a simple argument that allows to determine  the value of $s^*(n, \ell, \mSet)$ up to a constant. The argument is via an iterative application of Daykin's inequality and is due to Sergei Kiselev.

\begin{proposition}
\label{proposition: weak lower bound}
Put $\sigma = \sum_{S \in \binom{[\ell]}{2}} \mSet_{S}$. Then there is a constant $C$ depending on $\ell$ and $\mSet$ such that
\begin{align*}
    \targetFunction \le C n^\sigma 2^n.
\end{align*}
\end{proposition}

\begin{proof}
Consider a collection of  families $\ff_1, \ldots, \ff_{\ell}$ satisfying the $\mSet$-overlapping property. We prove that
\begin{align}
\label{eq: Daykin induction}
    \prod_{k = 1}^\ell |\ff_k|
    \le
%    |\ff_1 \wedge \ff_2|
    \prod_{k = 1}^{\ell-1}
        \left |
            \left (
                \bigvee_{k' = 1}^k
                    \ff_{k'}
            \right )
            \wedge
            \ff_{k + 1}
        \right| \cdot
    \left |
        \bigvee_{k = 1}^\ell
            \ff_k
    \right |
\end{align}
by induction on $\ell$. The statement is true for $\ell = 2$ due to Theorem~\ref{theorem: Daykin's inequality}. Suppose that $\ell\ge 3$ and that the statement holds for $\ell - 1$. Then
\begin{align*}
    \prod_{k = 1}^\ell
        |\ff_k|
    \le
    \left [
%       |\ff_1 \wedge \ff_2| \cdot
        \prod_{k = 1}^{\ell - 2}
            \left |
                \left (
                    \bigvee_{k' = 1}^k
                        \ff_{k'}
                \right )
                \wedge
                \ff_{k + 1}
            \right| \cdot
        \left |
            \bigvee_{k = 1}^{\ell - 1}
                \ff_k
        \right |
    \right ]
    |\ff_\ell|.
\end{align*}
Due to Theorem~\ref{theorem: Daykin's inequality}
\begin{align*}
    \left |
        \bigvee_{k = 1}^{\ell - 1}
            \ff_k
    \right | \cdot
    |\ff_\ell|
    \le
    \left |
        \left (
            \bigvee_{k = 1}^{\ell - 1}
            \ff_k
        \right )
        \wedge
        |\ff_\ell|
    \right | \cdot
    \left |
        \bigvee_{k = 1}^\ell
            \ff_k
    \right |,
\end{align*}
which proves~\eqref{eq: Daykin induction}. Next we bound its factors. Obviously,
\begin{align*}
    |\ff_1 \bigvee \ff_2|
    \le \left | \binom{[n]}{\le \mSet_{1, 2}} \right |
\end{align*}
More generally, for any $k \in[2, \ell - 1]$ we have
\begin{align*}
    \left |
        \left (
            \bigvee_{k' = 1}^k
                \ff_{k'}
        \right )
        \wedge
        \ff_{k + 1}
    \right|
    \le
    \left |
        \binom{[n]}{
            \le \sum_{k' = 1}^{k}\mSet_{k', k + 1}
        }
    \right |.
\end{align*}
Finally,
\begin{align*}
    \left |
        \bigvee_{k = 1}^\ell
            \ff_k
    \right |
    \le 2^n.
\end{align*}
Substituting these bounds in \eqref{eq: Daykin induction}, we derive the statement of the proposition. \end{proof}

\subsection{Degrees}

In what follows, we will be extensively working with the degrees of elements w.r.t. $\ff_1,\ldots,\ff_\ell$. We use the notion of the {\it normalized degree} of an set $F$, defined as follows:%A powerful tool which will be used here and further is a degree of element $x$ due to a family $\ff$. We define it in the following way:
\begin{align*}
    d(F, \ff) = \frac{
        |\ff(F)|
    }{
        |\ff|
    },
\end{align*}
%where $\ff(x) = \{F\setminus\{x\}\mid F\in \ff, x \in F\}$.
For brevity, we write $d_k(F)$ instead of $d(F, \ff_k)$ and $d_k(x)$ instead of $d_k(\{x\})$. %The degree $d_k(F)$ of a set $F$ is defined analogously.

%The first statement which describes its properties is the following:

\begin{proposition}
\label{proposition: big x fraction}
Let $\ff_1, \ldots, \ff_\ell$ be a collection of $\mSet$-overlapping families that is extremal. Then there is a set $I$ of size $O(\log n)$, such that for each $x \in [n] \setminus I$ there is $k\in [\ell]$ such that $d_{k}(x) \ge \frac 13$.
\end{proposition}
\begin{proof}
Fix some $k\in[\ell]$ and take a uniformly random set $X\in \ff_k$. Let $\mathbf v = (\mathbf v_1,\ldots, \mathbf v_n)$ be a random variable equal to the characteristic vector of $X.$ %charaConsider characteristic vectors of sets of $\ff_k$ for some $k$ on the Boolean cube. Denote the set of these vectors by $v(\ff_k)$. Furthermore, consider an uniformly distributed random vector $X=(X_1, \ldots. X_n)$ on the set $v(\ff_k)$.
For a subset $S$ of $[n]$ we denote $\left ( \mathbf v_i \right )_{i \in S}$ by $\mathbf v_S$. Thus, due to Claim~\ref{claim: entropy properties} (i)
\begin{align*}
    \entropy [\mathbf v] \le
    \sum_{i \in M_k}
    \entropy \left [
        \mathbf v_i
    \right ]
    + \entropy \left [ \mathbf v_{[n] \setminus \setCenter_k} \right].
\end{align*}
It easy to see that $\mathbf v_i = 1$ with probability $d_k(i)$. Moreover, by Claim~\ref{claim: entropy properties} (ii) we have %since the support of $\mathbf v_{[n] \setminus \setCenter_k}$ is $v \left (\ff_k |_{[n] \setminus \setCenter_k} \right )$  we obtain
$ \entropy [\mathbf v_{[n] \setminus \setCenter_k}] \le \log_2 \left  | \ff_k |_{[n] \setminus \setCenter_k} \right |$. Obviously, $\left |\ff_k |_{[n] \setminus \setCenter_k} \right | \le 2^{|\residualSet|} \left | \ff_k |_{\bigcup_i \setCenter_i \setminus \setCenter_k} \right |$ and, thus, $\log_2 \left  | \ff_k |_{[n] \setminus \setCenter_k} \right | \le |\residualSet| + O(\log n) = O(\log n)$ due to Lemma~\ref{lemma: log residual}. Because $\entropy [\mathbf v] = \log_2 |\ff_k| \ge |\setCenter_k|$, we observe
\begin{align*}
    |\setCenter_k|
    \le
    \sum_{i \in \setCenter_k}
        h_2 \left (
            d_k(x)
        \right )
    +
    O(\log n),
\end{align*}
where $h_2(p) = - p \log_2 p - (1-p) \log_2 (1-p)$ is the binary entropy.
Using the fact that $|\bigcup_{k\ne k'} \setCenter_k\cap \setCenter_{k'}|\le {\ell\choose 2} = O(1)$, we get
\begin{align*}
    \left |
        \bigcup_{k = 1}^\ell
            \setCenter_k
    \right |\le\sum_{k=1}^\ell |\setCenter_k|
    \le
    \sum_{i \in [n] \setminus \residualSet}
        \max_k h_2 \left (
            d_k(i)
        \right )
    +
    O(\log n).
\end{align*}
Using Lemma~\ref{lemma: log residual}, we get
\begin{align*}
    n \le
    \sum_{i = 1}^n
        \max_k
        h_2 \left (
            d_k(i)
        \right )
    +
    O(\log n)
\end{align*}
%due to Lemma~\ref{lemma: log residual}.
For each real-valued $\varepsilon\in [0,1]$ define $I_{\varepsilon}$ as
\begin{align*}
    I_\eps :=
    \left \{
        i \in [n]
        \mid
        \max_k
        h_2 \left (
            d_k(i)
        \right )
        < \varepsilon
    \right \}.
\end{align*}
Since $h_2(p) \le 1$ for any $0\le p\le 1,$ we obtain
\begin{align*}
     n - O(\log n) & \le \varepsilon |I_\varepsilon| + (n - |I_\varepsilon|),\\
    |I_\varepsilon| & \le \frac{O(\log n)}{1 - \varepsilon}.
\end{align*}
Note that $h_2(1/3) = \log_2 3-2/3$. Putting $\varepsilon = \log_2 3 - 2/3$ in the expression above, we get that for each $x\in[n] \setminus I_{\varepsilon}$ there is a family of sets $\ff_k$ such that $d_k(x) \ge \frac{1}{3}$ and that $|I_{\varepsilon}|=O(\log n)$.
\end{proof}

In addition, we observe the following property of degrees:

\begin{proposition}
\label{proposition: multiplicative property of degrees}
Let $\ff_1, \ldots, \ff_\ell$ be a collection of $\mSet$-overlapping families that is extremal. Then for any subset of indices of $K \subset [\ell]$ and sets $F_k \in \ff_k$, $k \in K$
\begin{align*}
    \prod_{k \in K} d_k(F_k) \le C_D
        n^{- \sum_{\{k, k'\} \in \binom{K}{2}} |F_k \cap F_{k'}| },
\end{align*}
where $C_D$ is some constant depending on $\mSet$, $\ell$.%, and $\sigma = \sum_{S \in \binom{[\ell]}{2}} \mSet_{S}$.
\end{proposition}

\begin{proof}
If for some $k, k'$ we have $F_{k}$, $F_{k'}$ such that $|F_{k} \cap F_{k'}| > \mSet_{k, k'}$, then either $d_{k}(F_k) =0$ or $d_{k'}(F_{k'}) = 0$ and the inequality is trivial. Thus, assume $|F_{k} \cap F_{k'}| \le \mSet_{k, k'}$ for each $k, k'$.
Consider families $\ff_k$, $k \in [\ell] \setminus K$, and $\ff_k(F_k)$, $k \in K$ as families in $2^{[n]}$. They satisfy the $\mSet'$-overlapping property, where
\begin{align*}
    \mSet_{S}' =
    \begin{cases}
    \mSet_{S}, & S \not \subset K \\
    \mSet_{S} - \left | \bigcap_{s \in S} F_s \right |, & S \subset K.
    \end{cases}
\end{align*}
According to Theorem~\ref{proposition: weak lower bound},
\begin{align*}
    \prod_{k \in [\ell] \setminus K} |\ff_k| \prod_{k \in K} |\ff_k(F_k)| = O \left (n^{\sigma - \sum_{\{k, k'\} \in \binom{K}{2}}  |F_k \cap F_{k'}|} 2^n \right ).
\end{align*}
Meanwhile, due to Theorem~\ref{theorem: general lower bound}
\begin{align*}
    \prod_{k \in [\ell]} |\ff_k| = \Theta(n^{\sigma} 2^n),
\end{align*}
and, consequently,
\begin{align*}
    \prod_{k \in K} d_k(F_k) = O(1) \cdot n^{- \sum_{\{k, k'\} \in \binom{K}{2}}  |F_k \cap F_{k'}| }.
\end{align*}
\end{proof}

\subsection{Proof}

Proposition~\ref{proposition: multiplicative property of degrees} implies the following Lemma which is crucial for the understanding the asymptotic of $\targetFunction$:

\begin{lemma}
\label{lemma: element excluding}
Let $\ff_1, \ldots, \ff_\ell$ be families from the extremal example. Then there are subfamilies $\ff_k' \subset \ff_k$, $k \in [\ell]$ such that
\begin{enumerate}
    \item Every element $x \in [n]$ is contained in at most two subfamilies $\ff_{k}'$'s.
    \item For every $k\in[\ell]$ it holds that $|\ff'_k| \ge (1 - \delta_n) |\ff_k|$, where $\delta_n = O(n^{-1/2})$.
\end{enumerate}
\end{lemma}

\begin{proof}
Given $x \in [n]$, consider two cases: $x \in I$ and $x \not \in I$, where $I$ is defined as in Proposition~\ref{proposition: big x fraction}. If $x \in [n] \setminus I$, then there is $k^* \in [\ell]$ such that $d_{k^*}(x) \ge \frac{1}{3}$. Consider any subset $K$ of $[\ell] \setminus \{k^*\}$ of cardinality $2$. From Proposition~\ref{proposition: multiplicative property of degrees} we get
\begin{align*}
    d_{k^*}(x) \prod_{k \in K} d_k(x) = O\big(n^{-3}\big).
\end{align*}
Thus, there is $k \in K$ such that
\begin{align*}
    d_k(x) = O\big(n^{-3/2}\big).
\end{align*}
It implies that there are $\ell - 2$ families in which $x$ has normalized degree $ O\big(n^{-3/2}\big)$. We put
\begin{align*}
    T_x = \left \{
        k
        \mid
        d_k(x)  =O\big(n^{-3/2}\big)
    \right \}.
\end{align*}

If $x \in I$, there may be no $k^* \in K$ such that $d_{k^*}(x) \ge 1/3$. Thus, consider an arbitrary set $K \subset [\ell]$ of cardinality $3$. Then
\begin{align*}
    \prod_{k \in K} d_k(x) = O\big(n^{-3}\big),
\end{align*}
and consequently, there is $k \in K$ such that
\begin{align*}
    d_k(x) =O\big( n^{-1}\big).
\end{align*}
Since $K$ is arbitrary, there are at least $\ell - 2$ families for which this inequality holds. We put
\begin{align*}
    S_x = \left \{
        k
        \mid
        d_k(x) = O\big( n^{-1}\big)
    \right \}.
\end{align*}
Put $W_k = \{x \mid k \in T_x \cup S_x \}$. Define $\ff_k' = \ff_k(\overline{W_k})$. Then
\begin{align*}
    \frac{
        \left |
            \ff_k'
        \right |
    }{
        |\ff_k|
    }
    & \ge 1 - \left (
        \left |
            \{ x \mid k \in S_x \}
        \right |
        O\big(n^{-1}\big)
        +
        \left |
            \{x \mid k \in T_x\}
        \right |
        O\big(n^{-3/2}\big)
    \right ) \\
    & \ge 1 - O\left (
        \frac{|I|}{n}
        + n^{-1/2}
    \right ) =: 1 - \delta_n.
\end{align*}
Since $|I| = O(\log n)$ due to Proposition~\ref{proposition: big x fraction}, we obtain $\delta = O(n^{-1/2})$.
\end{proof}

We are ready to prove the upper bound. %The following lemma provides us with asymptotics of $\targetFunction$ utilizing families $\ff_k'$ from Lemma~\ref{lemma: element excluding}.

\begin{proof}[Proof of the upper bound in Theorem~\ref{theorem: maximal families structure}]
Let $\ff_{1}', \ldots, \ff_\ell'$ be the families from Lemma~\ref{lemma: element excluding}. Then
\begin{align}
\label{eq: primed upper bound}
    \prod_{k = 1}^l |\ff_k| \le (1 - \delta_n)^{- \ell} \prod_{k = 1}^l |\ff_k'|
\end{align}
and each $\{x\}$, $x \in [n]$, is contained in the sets from at  most two families $\ff_k'$. Thus, for each $S \subset [l]$ with $|S| > 2$ we have
\begin{align}
    \bigwedge_{k \in S} \ff_k' = \varnothing.
\end{align}
Consequently, the sets $\support (\ff_k \land \ff_{k'})$ are disjoint for different pairs $\{k, k'\}$, where $\support \ff = \{x \in [n] \mid \{x\} \in \ff\}$. Hence, we can use Corollary~\ref{corollary: rinott sets theorem} and obtain
\begin{align*}
    \prod_{k = 1}^l |\ff_k'|
    \le
    \left |
        \bigvee_{1 \le k < k' \le l}
            (\ff_k' \land \ff_{k'}')
    \right |
    \left |
        \bigvee_{k = 1}^l
            \ff_k'
    \right |
    \le
    \prod_{1 \le k < k' \le \ell}
        \binom{
            |\support (\ff_k' \land \ff_{k'}')|
        }{
            \le \mSet_{k, k'}
        }
    2^n
\end{align*}
where the last inequality is due to the following obvious fact:
\begin{align*}
    \ff_k' \land \ff_{k'}' \subset \binom{\support (\ff_k' \land \ff_{k'}')}{\le \mSet_{k, k'}}.
\end{align*}

Optimizing over the choices for cardinalities of supports of $\ff'_k \wedge \ff_k$ leads us to the optimization problem~\eqref{eq: optimization problem}. As a reminder, it is formulated as follows:
\begin{align*}
    \solution & =
    \max_{\nv}
    \prod_{S \in \binom{[\ell]}{2}}
        \binom{
            \nv_S
        }{
            \le t
        } \\
    & \text{ s.t. }
    \sum_{S \in \binom{[\ell]}{2}}
        \nv_S
    = n.
\end{align*}

The asymptotics of the solution was obtained in the proof of Theorem~\ref{theorem: general lower bound}. Substituting the obtained bound on $\solution$ into~\eqref{eq: primed upper bound} concludes the proof.
\end{proof}

%\section{Conclusion}

%Firstly, we obtained the asympotics of $\targetFunction$. We applied a number of techniques, including  different correlation inequalities, entropy, Hoeffding inequality and some hypergraph theory. We have developed some tricks and methods that allowed us to obtain a number of structural results. Thanks to that, we managed to show the asymptotics of $\targetFunction$ and describe extremal examples.

\section{Acknowledgement}
We would like to thank Peter Frankl for sharing the proof of Theorem~\ref{theorem: Frankl's theorem} and Sergei Kiselev for actively participating in the project on its early stages, and in particular for suggesting the proof of Theorem~\ref{proposition: weak lower bound}. The research was in part supported by the RSF grant N 22-11-00131.

\end{document}